\documentclass[11pt]{amsart}
\usepackage{amsmath}
\usepackage{amsthm,amscd,amssymb,amsfonts,amsbsy}
\usepackage{latexsym}

\usepackage{color,graphicx}
\usepackage{enumerate}
\usepackage[pdftex,colorlinks]{hyperref}
\usepackage{amsfonts}
\numberwithin{equation}{section}

\theoremstyle{plain}
\newtheorem{theorem}{Theorem}
\newtheorem{lemma}{Lemma}
\newtheorem{corollary}{Corollary}
\newtheorem{proposition}{Proposition}

\theoremstyle{definition}

\theoremstyle{remark}

\newtheorem{remark}{Remark}
\newtheorem{notation}{Notation}

\newcommand{\mysection}[1]{\section{#1}
\setcounter{equation}{0}}

\newcommand{\De}{\Delta}

\newcommand{\Gam}{\Gamma}

\newcommand{\ep}{\varepsilon}

\newcommand{\Om}{\Omega}
\newcommand{\om}{\omega}

\newcommand{\si}{\sigma}

\newcommand{\bB}{\mathbb B}
\newcommand{\bE}{\mathbb E}
\newcommand{\bH}{\mathbb H}
\newcommand{\bL}{\mathbb L}
\newcommand{\bN}{\mathbb N}
\newcommand{\bP}{\mathbb P}
\newcommand{\bR}{\mathbb R}
\newcommand{\bZ}{\mathbb Z}

\newcommand{\cB}{\mathcal{B}}

\newcommand{\cF}{\mathcal{F}}

\newcommand{\cP}{\mathcal{P}}

\newcommand{\cS}{\mathcal{S}}

\newcommand{\ip}[1]{\left\langle#1\right\rangle}

\providecommand{\set}[1]{\{#1\}}
\providecommand{\Set}[1]{\left\{#1\right\}}

\providecommand{\Abs}[1]{\left\lvert#1\right\rvert}

\providecommand{\Bigabs}[1]{\Bigl\lvert#1\Bigr\rvert}

\providecommand{\norm}[1]{\lVert#1\rVert}
\providecommand{\Norm}[1]{\left\lVert#1\right\rVert}

\newcommand{\supp}{\operatorname{supp}}

\newcommand{\Rd}{{\bR^d}}
\newcommand{\RT}{\mathbb{R}_T^d}
\newcommand{\intRd}{\int_{\bR^d}}
\newcommand{\intT}{\int_0^T }

\newcommand{\sjoi}{\sum_{j=1}^\infty}
\newcommand{\sjz}{\sum_{j=-\infty}^\infty}

\newcommand{\f}{\frac}
\newcommand{\p}{\partial}

\begin{document}
\title{Stochastic heat equations driven by L\'evy processes}

\author[T. Chang]{TongKeun Chang}
\address{T. Chang: Department of Mathematics, Yonsei University, Seoul 120-749, Republic of Korea}
\email{chang7357@yonsei.ac.kr}

\author[M. Yang]{Minsuk Yang}
\address{M. Yang: Department of Mathematics, Yonsei University, Seoul 120-749, Republic of Korea}
\email{kusnim@gmail.com}

\thanks{The first author was supported  by the National Research Foundation of
Korea(NRF-2010-0016699)}

%\subjclass[2000]{Primary 60H15, Secondary; 35R60}
%\keywords{Stochastic heat equation, L\'evy process, Sobolev space,  Besov space.}

\begin{abstract}
We study stochastic heat equations driven by a class of L\'evy
processes:
\begin{equation*}
du = \De u dt + g dX_t \quad \mbox{in} \quad     \bR^d_T, \qquad u(0,x)= 0
\quad \mbox{in} \quad  x \in \bR^d.
\end{equation*}
We prove the corresponding estimate
\[\norm{u}_{\bH_p^k(\RT)} \le c(p,T) \norm{g}_{\bB_p^{k-\frac2p}(\RT)}\]
for $2\le p<\infty$ and $k \in \bR$.

\vspace{5mm}

\end{abstract}

\maketitle

%------------------------------------------------------------------------------%
\mysection{Introduction}                          \label{Section1}
%------------------------------------------------------------------------------%
In this paper, we study the following stochastic heat equation
\begin{equation}\label{spde1}
\begin{cases}
du  = (\De u  + f )dt + g  dX_t & \RT \vspace{2mm}\\
u|_{t=0} = u_0 & \Rd,
\end{cases}
\end{equation}
where $\bR^d_T : = (0,T) \times \Rd$ for $0 < T < \infty$.
We assume that $X_t$ is a one-dimensional L\'evy process satisfying some conditions on a probability space, which is explained in Section \ref{Section2}.
We allow $f$ and $g$ to be random.
To solve the problem \eqref{spde1}, we may consider the following two problems
\begin{equation}\label{spde2}
\begin{cases}
du  = (\De u  + f )dt & \RT \vspace{2mm}\\
u|_{t=0} = u_0 & \Rd
\end{cases}
\end{equation}
and
\begin{equation}\label{spde3}
\begin{cases}
du  = \De u  dt + g  dX_t &  \bR^d_T \vspace{2mm}\\
u|_{t=0} = 0 & \Rd.
\end{cases}
\end{equation}
Since the problem (\ref{spde2}) has been well studied, we shall focus on the problem (\ref{spde3}).
For $g \in L^p(0,T; {\mathcal S}^{'} (\Rd) )$ (where ${\mathcal S}^{'} (\Rd)$
 is the space of tempered distributions), the solution of (\ref{spde3})
can be represented by
\begin{align}\label{solution}
u(t,x) = \int_0^t T_{t-s} g(s,x) dX_s.
\end{align}
Here, $T_{t-s} g(s,x) =\Gamma(t-s,\cdot) *g(s,\cdot)(x) $, where $\Gamma(t,x) = (2\pi t)^{-\f{d}{2}} e^{-\frac{|x|^2}{4t}}$ and
$*$ is the standard convolution in $\Rd$.

For the Brownian motion case, a theory was developed by N.V. Krylov \cite{Kr94}.
Since the Burkholder-Davis-Gundy inequality implies
\[\bE \int_0^T\norm{ \nabla u(s, \cdot)}_{L^p(\bR^d)}^p ds
\leq c(p) \bE \intT \intRd \Big(\int_0^t
|\nabla T_{t-s}g(s,x)|^2 ds\Big)^{p/2} dxdt,\]
he showed that for $2\le p<\infty$ there is a positive constant $c(p)$ independent of $T$ such that
\begin{align}\label{Krylov}
\bE \intT \intRd \Big(\int_0^t  |\nabla T_{t-s}g(s,x)|^2 ds\Big)^{p/2} dxdt
\leq c(p) \bE  \int_0^T\|g(s,\cdot)\|^p_{L^p(\bR^d)}ds.
\end{align}
He proved this inequality by interpolating $L^2$ estimates via
Plancherel's theorem and sophisticated BMO estimates. Using the
properties of Sobolev spaces, it was generalized for $k\in \bR$
\begin{align}\label{0521-1}
\bE \int_0^T \| u(s, \cdot)\|^p_{ H^k_p(\bR^d)} ds
\leq c(p,T) \bE \int_0^T\| g(s, \cdot)\|^p_{H_p^{k-1}(\bR^d)} ds.
\end{align}
Here, the function space $H^{k}_p(\bR^d)$ is the usual Sobolev space
(see Section \ref{Section2}).

For the general L\'evy process case, a few results are known for these types of Sobolev estimates.
In this case, instead of the Burkholder-Davis-Gundy inequality Kunita's inequality is applicable and it produces
\begin{align}\label{0521-2}
\bE \int_0^T \norm{\nabla u(s, \cdot)}_{L^p(\bR^d)}^p ds \leq
& c\bE \intT \intRd \Big(\int_0^t |\nabla T_{t-s}g(s,x)|^2 ds\Big)^{p/2} dxdt \\
\notag&\quad+ c\bE \intT \intRd \int_0^t |\nabla T_{t-s}g(s,x)|^p dsdxdt.
\end{align}
The first term on the right-hand side is the same as in \eqref{Krylov},
but the second term on the right-hand side is new.
Recently, Z. Chen and K. Kim \cite{CK}  proved that for
$2 \le p<\infty$ and $\epsilon >0$, there is a constant
$c(\epsilon,p,T) > 0$ such that
\begin{equation}\label{ep}
 \bE \intT \intRd \int_0^t  |\nabla T_{t-s}g(s,x)|^p dsdxdt \leq
c(\epsilon,p,T)  \bE \intT \| g(s,\cdot)\|^p_{  H_p^{1 -\frac2p +\ep}(\bR^d)} ds
\end{equation}
under some assumptions on L\'evy measure.

Now we state our main results.

\begin{proposition}\label{prop}
Let $0<T<\infty$ and $1<p<\infty$.
There are positive constants $c_1(p,T)$ and $c_2(p)$ such that
\begin{equation}\label{no ep1}
\bE \intT \intRd \int_0^t  |T_{t-s}g(s,x)|^p dsdxdt \leq c_1(p,T)
\bE \intT \| g(s,\cdot)\|^p_{  B_p^{ -\frac2p}(\bR^d)} ds
\end{equation}
and
\begin{equation}\label{no ep2}
\bE \intT \intRd \int_0^t |T_{t-s}g(s,x)|^p dsdxdt \leq
c_2(p)  \bE \intT \norm{g(s,\cdot)}^p_{\dot B_p^{ -\frac2p}(\bR^d)} ds.
\end{equation}
\end{proposition}

To prove Proposition \ref{prop}, we shall use the Littlewood-Paley
theory and then prove variants of Hardy's inequality. Using
\eqref{0521-1}, \eqref{0521-2}, Proposition \ref{prop} and the
mapping properties of the pseudo-differential operators
$(I-\De)^{s/2}$ and $(-\De)^{s/2}$ (see (1) and (2) of  Remark
\ref{rem}) , we can obtain our main theorem.

\begin{theorem}\label{thm}
Let $0< T < \infty$ and $2\le p<\infty$ .
If $\beta_2<\infty$ and $\beta_p < \infty$, then
there are positive constants $c_1(p,T)$ and $c_2(p)$  such that
\begin{align*}
\norm{u}_{\bH_p^k(\RT)} &\le  c_1(p,T) \norm{g}_{\bB_p^{k-\frac2p}(\RT)}, \\
\norm{u}_{\dot \bH_p^k(\RT)} &\le  c_2(p) \norm{g}_{\dot \bB_p^{k-\frac2p}(\RT)},
\end{align*}
where $\beta_p$ is defined in \eqref{beta} and stochastic Banach spaces
$\bH_p^k(\RT), \,\, \bB_p^{k }(\RT), \,\, \dot \bH_p^k(\RT)$ and $ \dot \bB_p^{k }(\RT)$
are defined in \eqref{s-f spaces}.
\end{theorem}

A direct consequence of Theorem \ref{thm} is the following corollary
which follows from the fact that ${\mathbb H}^k_p(\bR^d)$ is
continuously embedded in ${\mathbb B}^k_p (\bR^d)$ for $k\in {\bR}$
and $2 \leq p < \infty$ and the property of real interpolation; (see
\cite{BL} Theorem 6.4.4 and Theorem 6.3.1).

\begin{corollary}\label{coro}
For $ 0< T < \infty$ and $2\le p<\infty$
\begin{align*}
\norm{u}_{\bH_p^k(\RT)} &\le c(p,T) \norm{g}_{\bH_p^{k-\frac2p}(\RT)}, \\
\norm{u}_{\bB_p^k(\RT)} &\le  c(p,T) \norm{g}_{\bB_p^{k-\frac2p}(\RT)}, \\
\norm{u}_{\dot \bH_p^k(\RT)} &\le  c(p) \norm{g}_{\dot \bH_p^{k-\frac2p}(\RT)}\\
\norm{u}_{\dot \bB_p^k(\RT)} &\le  c(p) \norm{g}_{\dot
\bB_p^{k-\frac2p}(\RT)}.
\end{align*}
\end{corollary}

The organization of the paper is as follows.
In Section 2, we introduce precise definitions of function spaces
and conditions concerning L\'evy processes.
In Section 3, we prepare basic lemmas for the heat kernel.
In Section 4, we reduce Proposition \ref{prop} to Lemma \ref{lemma3}.
In Section 5, we prove our main Lemma \ref{lemma3}.
In Section 6, we prove Theorem \ref{thm}.
In Section 7, we apply our method to SPDE with fractional Laplace operator.
%------------------------------------------------------------------------------%
\mysection{Preliminaries}                           \label{Section2}
%------------------------------------------------------------------------------%
\subsection{Sobolev and Besov spaces}
%------------------------------------------------------------------------------%
Let $ \cS:=\cS(\Rd)$ denote the class of Schwartz functions on $\Rd$.
The space $\cS':=\cS'(\Rd)$ is the dual space, i.e., the space
of continuous linear functionals on $\cS $.
Given $f\in\cS $, we define the Fourier transform and the inverse Fourier transform of $f$ by
\[\cF(f)(\xi) = \widehat{f}(\xi) = \intRd e^{-2\pi i \xi \cdot x} f(x) dx,
\quad \cF^{-1}(f)(x) = \intRd e^{2\pi i x\cdot\xi} f(\xi) d\xi.\]
The definition of Fourier transform is naturally extended to a tempered
distribution $f$; (see chapter 9 in \cite{Fo}).
We define the operators
\begin{align*}
(I-\De)^{k/2}f &= \cF^{-1}((1+4\pi^2|\xi|^2)^{k/2}\widehat{f}),\\
(-\De)^{k/2}f &= \cF^{-1}((2\pi|\xi|)^k\widehat{f})
\end{align*}
for $k\in\bR$ and for $f\in\cS(\Rd)$.
Let $k\in\bR$ and $1<p<\infty$.
The (nonhomogeneous) Sobolev space $H_p^k(\Rd)$ is defined as
\[H_p^k(\Rd) = \Set{f\in\cS'~|~\norm{f}_{H_p^k }:=\norm{(I-\De)^{k/2} f}_{L^p } < \infty},\]
and the homogeneous Sobolev space $\dot{H}_p^k(\Rd)$ is defined as
\[\dot{H}_p^k(\Rd) = \Set{f\in\cS'/\cP~|~\norm{f}_{\dot{H}_p^k }:=\norm{(-\De)^{k/2} f}_{L^p } < \infty},\]
where $\cS'/\cP$ denote the set of all tempered distributions modulo polynomials.
Note that to avoid working with equivalence classes of functions we identify two distributions in $\dot{H}_p^k(\Rd)$ whose difference is a polynomial.

Before we give the definition of Besov spaces, we prepare the setup.
We fix a function $\psi\in\cS(\Rd)$ satisfying $\widehat{\psi}(\xi)=1$ for $|\xi|\le1$ and $\widehat{\psi}(\xi)=0$ for $|\xi|\ge2$ and then define
$\widehat{\phi}(\xi)=\widehat{\psi}(\xi)-\widehat{\psi}(2\xi)$.
Note also that
\begin{equation}\label{support}
\supp\widehat{\phi}(\xi) \subset \set{1/2\le|\xi|\le2}.
\end{equation}
We define for $j\in\bZ$
\begin{equation}\label{phij}
\widehat{\phi}_j(\xi) = \widehat{\phi}(2^{-j}\xi)
\end{equation}
so that for all $\xi \in \bR^d$
\begin{equation}\label{partition1}
1=\widehat{\psi}(\xi)+\sjoi \widehat{\phi}_j(\xi)
\end{equation}
and for all $\xi\neq0$
\begin{equation}\label{partition2}
1=\sjz \widehat{\phi}_j(\xi).
\end{equation}
Let $k\in\bR$ and $1 \le p\le\infty$.
The (nonhomogeneous) Besov space $B_p^k(\Rd)$ is defined as
\[B_p^k(\Rd) = \Big\{f\in\cS'~|~\norm{f}_{B_p^k}:=\norm{\psi*f}_{L^p} +
\Big(\sjoi (2^{kj} \norm{\phi_j*f}_{L^p})^p\Big)^{1/p}<\infty\Big\},\]
and the homogeneous Besov space $\dot{B}_p^k(\Rd)$ is defined as
\[\dot{B}_p^k(\Rd) = \Big\{f\in\cS'/\cP~|~\norm{f}_{\dot{B}_p^k}: =
\Big(\sjz (2^{kj} \norm{\phi_j*f}_{L^p})^p\Big)^{1/p}<\infty\Big\},\]
where $*$ denotes the standard convolution in $\Rd$.
We note that whenever $\phi\in\cS$ and $f\in\cS'$, $\phi*f$ is a well defined function.

\begin{remark}\label{rem}
\begin{itemize}
\item[(1)]
For all $k, s\in\bR$, the pseudo-differential operator $(I-\De)^{s/2}$ is
isomorphism from $H_p^k(\Rd)$ to $H_p^{k-s}(\Rd)$ and from
$B_p^k(\Rd)$ to $B_p^{k-s}(\Rd)$.
\item[(2)]
For all $k, s\in\bR$, the pseudo-differential operator $(-\De)^{s/2}$ is
isomorphism from $\dot H_p^k(\Rd)$ to $\dot H_p^{k-s}(\Rd)$ and from
$\dot B_p^k(\Rd)$ to $\dot B_p^{k-s}(\Rd)$.
\item[(3)]
In particular, if  $1<p<\infty$ and $k$ is a nonnegative integer, then
$H_p^k(\Rd)$ is the set of functions satisfying
\[\sum_{0\le|\alpha|\le k} \intRd |\p^\alpha f(x)|^p dx < \infty,\]
where $\alpha = (\alpha_1, \alpha_2, \cdots , \alpha_d) \in (\bN\cup\set{0})^d$
and $\p^\alpha f = \p_{x_1}^{\alpha_1} \p_{x_2}^{\alpha_2} \dots \p_{x_d}^{\alpha_d} f$
is a distributional derivative.
\end{itemize}
\end{remark}
%------------------------------------------------------------------------------%
\subsection{Stochastic Banach spaces}
%------------------------------------------------------------------------------%
Let $(\Om,\cF,\set{\cF_t},\bP)$ be a probability space, where $\set{\cF_t : t\ge0}$ is a filtration of $\si$-fields $\cF_t \subset \cF$
with $\cF_0$ containing all $\bP$-null subsets of $\Om$.
Assume that a one-dimensional $\set{\cF_t}$-adapted L\'evy processes $X_t$ is defined on $(\Om,\cF,P)$.
We denote the expectation of a random variable $X(\omega)$, $\om\in\Om$
by $\bE[X]$ or simply $\bE X$.
We consider $g$ as a Banach space-valued stochastic process and so $(\Omega\times(0,T),\mathcal{P},P\bigotimes\ell((0,T]))$ is a
suitable choice for their common domain, where $\cP$ is the predictable $\si$-field
generated by $\set{\cF_t : t \geq 0}$ (see, e.g., pp. 84--85 of \cite{K1}) and $\ell((0,T])$ is the Lebesgue measure on $(0,T)$.
We define the stochastic function space
\begin{equation}\label{s-f spaces}
\bH_p^k(\RT)=L^p(\Om\times(0,T),{\mathcal P},H^k_p(\Rd))
\end{equation}
with the norm
\[\norm{f}_{\bH_p^k(\RT)}=\left(\bE\intT \norm{f(s,\cdot)}_{H_p^k(\Rd)}^pds \right)^{1/p}.\]
The stochastic function spaces $\dot\bH_p^k(\RT)$, $\bB_p^k(\RT)$ and $\dot\bB_p^k(\RT)$ are defined similarly.
%------------------------------------------------------------------------------%
\subsection{L\'evy process}
%------------------------------------------------------------------------------%
A Levy process $X_t:=X(t)$ is a stochastic process satisfying
\begin{itemize}
\item[(L1)]
$X(0)=0$ a.s.,
\item[(L2)]
$ X(t)$ has stationary and independent increments,
\item[(L3)]
$X(t)$ is stochastically continuous, i.e. for all $a>0$ and for all $s\ge0$,
\[\lim_{t\to s} \bP(|X(t)-X(s)|>a)=0.\]
\end{itemize}
A process $X_t$ is c\'adl\'ag if $X_t$ has left limit and is right continuous.
Since every L\'evy process has a c\'adl\'ag modification
that is itself a L\'evy process (see Theorem 2.1.8 in \cite{App}),
we may assume all L\'evy processes $X_t$ are c\'adl\'ag.
Let $t \geq 0$ and Borel sets $A \in \cB(\Rd\setminus\set{0})$.
We denote
\[N(t,A) = \#\set{0\le s \le t : X(s)-X(s-)\in A},\]
the intensity measure $\nu(A)=\bE[N(1,A)]$, and the compensated Poisson random measure
\[\widetilde{N}(t,A) = N(t,A)-t\nu(A).\]
Note that $\nu(A)$ is the L\'evy measure of $X_t$.
By the L\'evy-Ito decomposition (see more details in \cite{App}), there exist a constant $c\in\Rd$ and a positive-definite matrix $A$ such that
\[X_t = ct + AB_t + \int_{|z|<1} z \widetilde{N}(t, dz) + \int_{|z|\ge1} z N(t, dz),\]
where $B_t$ is a $d$-dimensional Brownian motion.
If we denote $\widetilde{c} := c + \nu(\{|z| \geq 1\})$, we may write
\begin{align*}
X_t = \tilde ct + AB_t + \int_{\bR^d} z \widetilde{N}(t, dz).
\end{align*}
Note that $\widetilde{N}(t,z) $ is martingale.
Since the result for the Brownian motion is known, we assume that $\widetilde{c}=0$ and $A=0$ for the simplicity.
Finally, we denote
\begin{equation}\label{beta}
\beta_p = \int_{\bR^d} |z|^p \nu(dz).
\end{equation}
%------------------------------------------------------------------------------%
\mysection{Basic Heat Kernel Estimates}                     \label{Section3}
%------------------------------------------------------------------------------%
We give basic lemmas for the heat kernel that will be useful in the sequel.

\begin{notation}
We denote $f \lesssim g$ if $f \le cg$ for some   positive constant $c$.
\end{notation}

\begin{lemma}\label{lemma1}
There exists a constant $c>0$ such that for all $j\in\bZ$
\[\Norm{\cF^{-1}(\widehat{\phi}_j(\xi)e^{-t|\xi|^2})}_{L^1} \lesssim e^{-ct2^{2j}},\]
where $\widehat{\phi}_j$ is defined in \eqref{phij}.
The implicit constant depends only on the dimension $d$.
\end{lemma}

\begin{proof}
Let $K_j(t,x)=\cF^{-1}(\widehat{\phi}_j(\xi)e^{-t|\xi|^2})(x)$.
By a simple scaling
\begin{align*}
K_j(t,x)
&=\intRd e^{2\pi ix\cdot\xi} \widehat{\phi}(2^{-j}\xi) e^{-t|\xi|^2} d\xi \\
&=2^{jd} \intRd e^{ 2\pi i2^jx\cdot\xi} \widehat{\phi}(\xi) e^{-2^{2j}t|\xi|^2} d\xi \\
&=2^{jd} K_0(2^{2j}t,2^jx).
\end{align*}
Observe that
\[(I-\De_\xi)e^{2\pi ix\cdot\xi}=(1+4\pi^2|x|^2)e^{2\pi ix\cdot\xi}.\]
Carrying out the repeated integrations by parts gives
\[(1+ 4\pi^2|x|^2)^N K_0(t,x) = \intRd e^{2\pi ix\cdot\xi} (I-\De_\xi)^N(\widehat{\phi}(\xi) e^{-t|\xi|^2}) d\xi\]
for all $N\in\bN$.
Since $\supp\widehat{\phi}(\xi) \subset \set{1/2\le|\xi|\le2}$, we have
\[(1+ 4\pi^2|x|^2)^N|K_0(t,x)| \lesssim
\sup_\xi|(I-\De_\xi)^N(\widehat{\phi}(\xi) e^{-t|\xi|^2})|.\]
A direct computation shows that for some $c>0$
\[\sup_\xi|(I-\De_\xi)^N(\widehat{\phi}(\xi) e^{-t|\xi|^2})| \lesssim e^{-ct}.\]
We choose $N>d$ so that
\[\intRd |K_j(t,x)| dx = \intRd |K_0(2^{2j}t,x)| dx \lesssim e^{-c2^{2j}t}.\]
This completes the proof.
\end{proof}

Given $g \in H_p^k(\Rd)$, we denote
\[T_t g(x) = \Gam(t,\cdot)*g(x).\]
In fact, for $k<0$, it is the convolution of a function with a tempered distribution, that is,
\[\Gam_t*g(x) =
\begin{cases}
\intRd \Gam(t,x-y)g(y) dy & k\ge0 \vspace{2mm}\\
\ip{g,\Gam(t,x - \cdot)} & k<0,
\end{cases}\]
where $\ip{\cdot,\cdot}$ means the duality paring between $\cS'(\Rd)$ and $\cS(\Rd)$.

\begin{lemma}\label{lemma2}
There exists a constant $c>0$ such that for all $j\in\bZ$
\[\norm{T_t(\phi_j*g)(s,\cdot)}_{L^p} \lesssim e^{-c2^{2j}t} \norm{\phi_j*g(s,\cdot)}_{L^p},\]
where $\phi*g(s,x):=\phi*g(s,\cdot)(x)$.
The implicit constant depends only on the dimension $d$.
\end{lemma}

\begin{proof}
We have
\[T_t(\phi_j*g)(s,x)=\cF^{-1}\big(e^{-t|\xi|^2} \widehat{\phi}_j(\xi) \widehat{g}(s,\xi)\big)(x).\]
From the support condition \eqref{support},
\[\widehat{\phi}_j(\xi)=\Big(\widehat{\phi}_{j-1}(\xi)+\widehat{\phi}_j(\xi)+\widehat{\phi}_{j+1}(\xi)\Big) \widehat{\phi}_j(\xi).\]
Thus we have
\[T_t(\phi_j*g)(s,x) = \sum_{-1\le k \le 1} \cF^{-1}(\widehat{\phi}_{j+k}(\xi)e^{-t|\xi|^2}) * \cF^{-1}(\widehat{\phi}_j(\xi)\widehat{g}(s,\xi))(x).\]
Young's convolution inequality gives
\[\norm{T_t(\phi_j*g)(s,\cdot)}_{L^p} \le \sum_{-1\le k \le 1} \norm{\cF^{-1}(\widehat{\phi}_{j+k}(\xi)e^{-t|\xi|^2})}_{L^1} \norm{\phi_j*g(s,\cdot)}_{L^p}\]
and therefore the result follows from Lemma \ref{lemma1}.
\end{proof}
%------------------------------------------------------------------------------%
\mysection{Proof of Proposition \ref{prop}}   \label{Section4}    %------------------------------------------------------------------------------%
First we consider (\ref{no ep1}).
Using \eqref{partition1} we decompose
\begin{equation}\label{decom}
g(s,x)= \psi*g(s,x) + \sjoi\phi_j*g(s,x).
\end{equation}
Using \eqref{decom} and Minkowski's inequality, we have
\begin{align*}
&\bE\intT \int_0^t \norm{T_{t-s}g(s,\cdot)}_{L^p}^p dsdt \\
&\lesssim \bE\intT \int_0^t \norm{T_{t-s}(\psi*g)(s,\cdot)}_{L^p}^p dsdt \\
&\quad + \bE\intT \int_0^t
\Big(\sjoi\norm{T_{t-s}(\phi_j*g)(s,\cdot)}_{L^p}\Big)^p dsdt.
\end{align*}
By Young's inequality, the first term of right-hand side is  dominated by
\begin{align*}
& \bE\intT \int_0^t \| \Gam(t-s, \cdot)\|^p_{L^1}
\norm{ \psi*g (s,\cdot)}_{L^p}^p dsdt
= \bE\intT \int_0^t \norm{\psi*g(s,\cdot)}_{L^p}^p dsdt.
\end{align*}
Using lemma \ref{lemma2}, the second term of right-hand side is dominated by
\begin{align*}
 \bE\intT \int_0^t
\Big(\sjoi e^{-c2^{2j}(t-s)} \norm{\phi_j*g(s,\cdot)}_{L^p}\Big)^p dsdt.
\end{align*}
Hence, we have
\begin{align*}
&\bE\intT \int_0^t \norm{T_{t-s}g(s,\cdot)}_{L^p}^p dsdt \\
& \qquad \lesssim    \bE\intT \int_0^t \norm{\psi*g(s,\cdot)}_{L^p}^p dsdt +
\Big(\sjoi e^{-c2^{2j}(t-s)} \norm{\phi_j*g(s,\cdot)}_{L^p}\Big)^p dsdt.
\end{align*}
If we denote
\begin{equation}\label{notation}
f_j(t-s):=e^{-c2^{2j}(t-s)} \quad \text{and} \quad
g_j(s):=\norm{\phi_j*g(s,\cdot)}_{L^p},
\end{equation}
then to prove \eqref{no ep1}, it suffices to show that
\begin{equation}\label{goal2}
\intT \int_0^t \Big(\sum_{1 \leq j < \infty} f_j(t-s) g_j(s)\Big)^p ds dt
\lesssim \intT \sum_{1 \leq j < \infty} 2^{-2j} g_j(s)^p ds.
\end{equation}

Now we consider \eqref{no ep2}.
Using \eqref{partition2}, we decompose
\[g(s,x)=\sum_{j =-\infty}^\infty \phi_j*g(s,x).\]
Using similar calculation with the above estimation, we obtain
\begin{align*}
&\bE\intT \int_0^t \norm{T_{t-s}g(s,\cdot)}_{L^p}^p dsdt \\
&\le \bE\intT \int_0^t \Big(\sjz\norm{T_{t-s}(\phi_j*g)(s,\cdot)}_{L^p}\Big)^p dsdt \\
&\le \bE\intT \int_0^t \Big(\sjz e^{-c2^{2j}(t-s)}
\norm{\phi_j*g(s,\cdot)}_{L^p}\Big)^p dsdt.
\end{align*}
Hence, to prove \eqref{no ep2}, it suffices to show that
\begin{equation}\label{goal2-2}
\intT \int_0^t \Big(\sjz f_j(t-s) g_j(s)\Big)^p ds dt
\lesssim \intT \sjz 2^{-2j} g_j(s)^p ds.
\end{equation}
We shall prove the inequalities \eqref{goal2} and \eqref{goal2-2} in Section \ref{Section5}.
%------------------------------------------------------------------------------%
\section{Proof of Main Lemma} \label{Section5}
%------------------------------------------------------------------------------%
\begin{lemma}\label{lemma3}
%Suppose there exists $c>0$ such that for all $j\in\bN$ and $t>0$
%\begin{equation}\label{decay}
%|f_j(t)| \lesssim e^{-c2^{2j}t}.
%\end{equation}
For $0 < T< \infty$ and  $1<p<\infty$
\begin{align}\label{L-goal2}
\intT \int_0^t \Bigabs{\sjoi f_j(t-s) g_j(s)}^p ds dt
\lesssim \intT \sjoi 2^{-2j} |g_j(s)|^p ds
\end{align}
and
\begin{align}\label{L-goal3}
\intT \int_0^t \Bigabs{\sjz f_j(t-s) g_j(s)}^p ds dt
\lesssim \intT \sjz 2^{-2j} |g_j(s)|^p ds.
\end{align}
\end{lemma}

\begin{proof}
We only prove \eqref{L-goal2} since the proof of \eqref{L-goal3} is almost the same.
%Without loss of generality we can assume that $f_j(t)$ and $g_j(t)$ are non-negative.
In order to use the decay of the function $f_j(t)$,
we separate the indices of the summation as
\begin{equation}\label{lem3-4}
\begin{split}
&2^{1-p} \intT \int_0^t \Big(\sjoi f_j(t-s) g_j(s)\Big)^p ds dt \\
&\le \intT \int_0^t \Big(\sum_{2^{2j}(t-s)\le1} f_j(t-s) g_j(s)\Big)^p ds dt
+ \intT \int_0^t \Big(\sum_{2^{2j}(t-s)>1} f_j(t-s) g_j(s)\Big)^p ds dt\\
&:= J_1 + J_2.
\end{split}
\end{equation}
If $2^{2j}(t-s) \leq 1$, then   $f_j(t-s) \leq c$ for some
positive constant $c$ depending only on $d$ . Hence, using H\"older's inequality, we have
\[J_1 \le \intT \int_0^t \Big(\sum_{2^{2j}(t-s)\le1} 2^{j/(p-1)}\Big)^{p-1} \sum_{2^{2j}(t-s)\le1} 2^{-j} g_j(s)^p ds dt.\]
Summing a geometric series, we have
\[\Big(\sum_{2^{2j}(t-s)\le1} 2^{j/(p-1)} \Big)^{p-1} \lesssim (t-s)^{-1/2}.\]
Changing the order of integration and summation, we get
\begin{equation}\label{lem3-5}
\begin{split}
J_1& \lesssim \intT \int_0^t (t-s)^{-1/2} \sum_{2^{2j}(t-s)\le1} 2^{-j} g_j(s)^p ds dt \\
&= \intT \sjoi 2^{-j} g_j(s)^p \int_s^{s+2^{-2j}} (t-s)^{-1/2} dtds \\
&\lesssim \intT \sjoi 2^{-2j} g_j(s)^p ds.
\end{split}
\end{equation}

Now, we estimate $J_2$. Let us fix   $2<r<2p$. Using H\"older's inequality, we obtain
\begin{align*}
J_2 &= \intT \int_0^t \Big(\sum_{2^{2j}(t-s)>1} 2^{rj/p} f_j(t-s) 2^{-rj/p} g_j(s)\Big)^p ds dt \\
&\lesssim \intT \int_0^t \Big(\sum_{2^{2j}(t-s)>1} 2^{rj/(p-1)} f_j(t-s)^{p/(p-1)}\Big)^{p-1} \sum_{2^{2j}(t-s)>1} 2^{-rj} g_j(s)^p ds dt.
\end{align*}
Since $f_j(t-s) \lesssim 2^{-2j}(t-s)^{-1}  $ for  $2^{2j}(t-s)>1$,
summing a geometric series, we have
\begin{align*}
\Big(\sum_{2^{2j}(t-s)>1} 2^{rj/(p-1)} f_j(t-s)^{p/(p-1)}\Big)^{p-1}
&\lesssim \Big(\sum_{2^{2j}(t-s)>1} 2^{j(r-2 p)/(p-1)}\Big)^{p-1} (t-s)^{-p}\\
&\lesssim (t-s)^{-r/2}.
\end{align*}
By changing the order of integration and summation, we get
\begin{equation}\label{lem3-6}
\begin{split}
J_2& \lesssim\intT \int_0^t (t-s)^{-r/2} \sum_{2^{2j}(t-s)>1} 2^{-rj} g_j(s)^p ds dt \\
&\le \intT \sjoi 2^{-rj} g_j(s)^p \int_{s+2^{-2j}}^\infty (t-s)^{-r/2} dtds \\
&\lesssim \intT \sjoi 2^{-2j} g_j(s)^p ds.
\end{split}
\end{equation}
From \eqref{lem3-4}, \eqref{lem3-5} and \eqref{lem3-6}, we obtain \eqref{L-goal2}.
\end{proof}
%------------------------------------------------------------------------------%
\section{Proof of Theorem \ref{thm}}          \label{Section6}
%------------------------------------------------------------------------------%
Since the proofs are similar, we only prove the first inequality.
Let $u$ be a function defined in \eqref{solution}.
Since we have
\[(I-\De)^{k/2}u(t,x)
= \int_0^t \ip{\Gam_{t-s},(I-\De)^{\frac{k}2}g(s,\cdot)} dX_s
= \int_0^t T_{t-s}( (I - \De)^{\frac{k}2} g )(s,x) dX_s,\]
it is sufficient to prove the case $k=0$, that is,
\begin{align}\label{thm0}
\norm{u}_{\bL^p(\RT)}
=\left(\bE\int_0^t \norm{u(s,\cdot)}_{L^p(\Rd)}^pds \right)^{1/p}
\lesssim \norm{g}_{\bB_p^{-\frac2p}(\RT)}.
\end{align}
Actually, if this estimate is proved, then, by (1) of  remark \ref{rem}, we obtain
\[\norm{(I-\De)^{k/2}u}_{\bL^p(\RT)} \lesssim \norm{(I-\De)^{k/2}g}_{\bB_p^{-\frac2p}(\RT)}\]
and  hence we have
\[\norm{u}_{\bH_p^k(\RT)} \lesssim \norm{g}_{\bB_p^{k-\frac2p}(\RT)}.\]
Using Kunita's inequality
(see pp. 332-335 in \cite{Ku}, corollary 4.4.24 in \cite{App}), we have
\begin{equation}\label{Kunita}
\begin{split}
\norm{u}_{\bL^p(\RT)}^p
&=\bE \intT \int_{\bR^d} \Abs{\int_0^t T_{t-s} g(s,\cdot)(x)dX_s}^pdxdt \\
&\lesssim \bE \intT \intRd \int_0^t \intRd |T_{t-s}g(s,x)|^p|z|^p \nu(dz) dsdxdt \\
&\quad+ \bE \intT \intRd \Big(\int_0^t \intRd |T_{t-s}g(s,x)|^2|z|^2 \nu(dz)ds\Big)^{p/2} dxdt\\
&= \beta_p \bE \intT \intRd \int_0^t  |T_{t-s}g(s,x)|^p dsdxdt \\
&\quad+ \beta_2^{\frac{p}2} \bE \intT \intRd \Big(\int_0^t |T_{t-s}g(s,x)|^2 ds\Big)^{p/2} dxdt.
\end{split}
\end{equation}
From \eqref{no ep1} and \eqref{Kunita}, we get \eqref{thm0}.
%------------------------------------------------------------------------------%
\mysection{SPDE with Fractional Laplace operator}  \label{Section7}
%------------------------------------------------------------------------------%
In this section, we give an application to the SPDE with fractional Laplace operator.
\begin{equation}\label{fspde1}
\begin{cases}
du  = -(-\De)^\alpha u dt + g dX_t & \RT \vspace{2mm}\\
u|_{t=0} = u_0 & \Rd,
\end{cases}
\end{equation}
where $(-\De)^\alpha u$, $0<\alpha<1$, is the fractional Laplacian of $u$ defined by
\begin{align}\label{definition2}
(-\De)^\alpha u(x) := c(d,\alpha) \int_{\bR^d} \frac{ u(x+y) - 2 u(x) + u(x-y)}{|y|^{n+2\alpha}} dy
\end{align}
with $c(d,\alpha)$ is a normalization constant. The fractional
Laplacian of $u$ also can be defined as a pseudo-differential
operator
\begin{equation}\label{definition3}
(-\De)^\alpha u(x) = \cF^{-1}((2\pi|\xi|)^{2\alpha}\widehat{u}(\xi))(x).
\end{equation}
The solution $u$ of \eqref{fspde1} is represented by
\[u(t,x) = \int_0^t P_{t-s} g(s,x) dX_s,\]
where $P_tg(s,x) = p(t, \cdot) * g(s, \cdot)(x)$ with
fundamental solution $p(t,x)$ of the fractional Laplace equation which is given by
\[p(t,x) = \cF^{-1}(e^{-t|\xi|^{2\alpha}})(x).\]
By a slight modification of the proof of Proposition \ref{prop},
one can prove the following estimate
\begin{proposition}\label{prop2}
Let $0<T<\infty$ and $2\le p<\infty$. There is a positive constant
$c$ such that
\[\bE \intT \intRd \int_0^t  |P_{t-s}g(s,x)|^p dsdxdt
\le  c\bE \int_0^T \|g(t,\cdot)\|^p_{B_p^{-\frac{2\alpha}p}(\bR^d)} dt.\]
\end{proposition}

A direct consequence is the following theorem.
\begin{theorem}\label{thm2}
For $2 \leq p < \infty$,
\begin{align*}
\| u\|_{{\mathbb H}_p^k(\bR^d_T)}
\leq c(p,T) \| g\|_{{\mathbb B}_p^{k- \frac{2\alpha}p}(\bR^d_T)}.
\end{align*}
\end{theorem}

\begin{proof}
We sketch the proof of Theorem \ref{thm2}.
From the same reasoning in the proof of Theorem \ref{thm}, we may assume that $k=0$.
Using the Kunita's inequality, we have for some $c>0$
\begin{align*}
\bE \int_0^T \norm{u(s, \cdot)}_{L^p(\bR^d)}^p ds \leq
&  c \beta_p\bE \intT \intRd \int_0^t  |P_{t-s}g(s,x)|^p dsdxdt \\
&\quad+c \beta_2^{\frac{p}2} \bE \intT \intRd \Big(\int_0^t |P_{t-s}g(s,x)|^2 ds\Big)^{p/2} dxdt.
\end{align*}
H. Kim and I. Kim\cite{KK} showed that for $2\le p<\infty$
\[\bE \intT \intRd \Big(\int_0^t |P_{t-s}g(s,x)|^2 ds\Big)^{p/2} dxdt \le c
\bE \int_0^T \| g(t, \cdot)\|^p_{H_p^{-1}(\bR^d)} dt\]
for some $c>0$ (see also \cite{CL}).
By the same proof as in Lemma \ref{lemma1} and Lemma \ref{lemma2}, one can obtain
\[\norm{P_t(\phi_j*g)(s,\cdot)}_{L^p}
\lesssim e^{-c2^{2j\alpha}t} \norm{\phi_j*g(s,\cdot)}_{L^p}.\]
Similar to the proof of Theorem \ref{thm}, we can obtain the result.
\end{proof}
%------------------------------------------------------------------------------%

\end{document}